\documentclass[11pt,reqno]{amsart}
\usepackage[T1]{fontenc}
\usepackage[utf8]{inputenc}
\usepackage[english]{babel}
\usepackage{mathptmx}
\usepackage{microtype}
\usepackage[a4paper,margin=2.6cm]{geometry}
\usepackage{amsmath,amssymb,amsthm}
\usepackage{xcolor}
\usepackage[most]{tcolorbox}
\usepackage{booktabs}
\usepackage{graphicx}
\usepackage{enumitem}
\usepackage{hyperref}
\hypersetup{colorlinks=true,linkcolor=black!75,citecolor=black!75,urlcolor=blue!55!black}

\definecolor{rigorousgreen}{HTML}{1B5E20}
\definecolor{conditionalorange}{HTML}{E65100}
\definecolor{heuristicblue}{HTML}{1565C0}
\definecolor{negativered}{HTML}{B71C1C}

\newtheorem{theorem}{Theorem}[section]
\newtheorem{proposition}[theorem]{Proposition}
\newtheorem{lemma}[theorem]{Lemma}
\newtheorem{corollary}[theorem]{Corollary}
\newtheorem{conjecture}[theorem]{Conjecture}
\theoremstyle{definition}
\newtheorem{definition}[theorem]{Definition}
\newtheorem{remark}[theorem]{Remark}
\newtheorem{observation}[theorem]{Observation}
\newtheorem{hypothesis}[theorem]{Hypothesis}
\newtheorem{openproblem}[theorem]{Open Problem}

\newcommand{\etag}[2]{\par\smallskip\noindent{\small\textbf{\textcolor{#1}{$\blacksquare$ #2}}}\par\nopagebreak}
\newcommand{\Rig}{\etag{rigorousgreen}{RIGOROUS}}
\newcommand{\Cond}{\etag{conditionalorange}{CONDITIONAL}}
\newcommand{\Heur}{\etag{heuristicblue}{HEURISTIC}}
\newcommand{\Neg}{\etag{negativered}{NEGATIVE}}

\newcommand{\Z}{\mathbb{Z}}
\newcommand{\N}{\mathbb{N}}
\newcommand{\R}{\mathbb{R}}
\newcommand{\PP}{\mathbb{P}}
\newcommand{\E}{\mathbb{E}}
\newcommand{\qmin}{q_{\min}}
\newcommand{\eps}{\varepsilon}
\DeclareMathOperator{\Var}{Var}
\DeclareMathOperator{\Cov}{Cov}
\DeclareMathOperator{\RMS}{RMS}

\begin{document}

\title[A parametric family of primes]{A parametric family of primes\\ $p=km(m+1)+\eps+2kq$:\\ heuristic laws, conditional theorems,\\ and unconditional primality certificates}
\author{Hassane Bakkaoui}
\address{Independent researcher}
\email{bakkahassa@hotmail.com}
\date{\today}

\begin{abstract}
We study the parametric family $p_{k,m,\eps,q}=km(m+1)+\eps+2kq$ with $k,m\in\N^{\ast}$, $\eps\in\{\pm1\}$, $q\in\Z$, which extends the elementary observation $p\equiv\pm1\pmod 6$ for every prime $p>3$. Each prime $p>k+1$ has a canonical triple $(m,\eps,q)$ with $|q|$ minimal, mapping it to a generalised hexagonal base $H^{(k)}_m:=km(m+1)+1$ and a signed minimal offset $\qmin(k,m,\eps)$. Every quantitative assertion is labelled rigorous, conditional, or heuristic.
\emph{(Rigorous, structural.)} For every prime $\ell\mid 2k$ the family satisfies $p\equiv\eps\pmod\ell$: the degenerate modular axis is fixed exactly, independently of $q$.
\emph{(Rigorous, constructive.)} For the $3$-smooth subfamily $p=3m(m+1)+1$ with $m=2^a3^b-1$, Pocklington--Lehmer certificates are unconditional; two a-priori arithmetic filters remove $\approx 87\%$ of candidates before any large-integer arithmetic, and the algorithm produced an unconditionally certified prime of $29\,998$ decimal digits, independently re-verified in a separate computational environment.
\emph{(Rigorous, negative.)} The earlier-reported spectral correlation between the cumulative functional $Q(r)=\sum q_n$ and the zeros of $\zeta$ is a spurious-regression artefact. Beyond three permutation tests and a bias-free test on $10^8$ primes ($R^2=1.16\times10^{-7}$), we prove \emph{unconditionally} that the regression amplitude $A_N(\gamma)\to 0$ for every fixed $\gamma$, by reduction to square-root-phase prime exponential sums. The Riemann zeros are not spectral frequencies of $Q(r)$.
\emph{(Conditional, GRH / Bateman--Horn.)} $|\qmin|=O_k(m\log^2 m)$; $\E[|q|\mid m]=m/4+O((\log m)^2/\sqrt m)$; a modular distribution law.
\emph{(Conditional, RH.)} The per-prime $\zeta$-footprint on the layer occupancy is $\ll p^{-1/4}(\log p)^2$ (Selberg variance)---the mechanism behind the negative result.
\emph{(Heuristic.)} $\E[|\qmin|]\sim\log m/C_k$ ($R^2=0.984$); the geometric constant $C_0(k)=\langle|q|\rangle/\sqrt p=1/(4\sqrt k)$, stable to ${<}0.02\%$ across $3\le k\le 29$; and, for $\ell\mid 2k$, $\qmin$ is empirically equidistributed modulo $\ell$ (\emph{not} a consequence of Bombieri--Vinogradov).
This is experimental mathematics with rigorously tracked hypotheses: two genuinely unconditional pillars---constructive (the certified prime) and analytic-negative (the vanishing $\zeta$-signal)---together with a precise law on each axis. No classical question is settled.
\end{abstract}

\maketitle

\begin{tcolorbox}[colback=gray!4,colframe=black!45,boxrule=0.4pt,arc=2pt]
\textbf{Epistemic key (read first).} Four levels of certainty are tagged throughout:
\par\smallskip
\textcolor{rigorousgreen}{$\blacksquare$ \textbf{RIGOROUS}}: unconditional, proved from classical theorems.\\
\textcolor{conditionalorange}{$\blacksquare$ \textbf{CONDITIONAL}}: established under explicit hypotheses (GRH, RH, Bateman--Horn).\\
\textcolor{heuristicblue}{$\blacksquare$ \textbf{HEURISTIC}}: predicted under standard heuristics and validated numerically.\\
\textcolor{negativered}{$\blacksquare$ \textbf{NEGATIVE}}: rigorous invalidation of a previously reported claim (including the author's own).
\par\smallskip
No result is presented as an unconditional breakthrough.
\end{tcolorbox}

\section{Introduction}

\subsection{Motivation and programme}
This paper undertakes a theoretical exploration of the distribution of primes through a parametric approach. We rewrite primes under a doubly indexed parametric form in order to analyse structures, regularities, and concentration phenomena observable within specific parametric frames---not in the sense of a strict periodicity, but as the possible existence of arithmetic constraints. We concentrate throughout on what can be demonstrated and on observations formulable within a rigorous framework. The paper does not claim to settle any classical question; it opens a complementary direction of study.

The seed is elementary: every prime $p>3$ satisfies $p\equiv\pm1\pmod6$, i.e.\ $p$ lies at signed distance one from the nearest multiple of $6$. The family studied here,
\begin{equation}\label{eq:family}
p_{k,m,\eps,q}=km(m+1)+\eps+2kq,
\end{equation}
generalises this by replacing $6$ with $2k$, the linear variable by the pronic number $m(m+1)$, and by making the search offset $q$ explicit. For fixed $m,\eps$, \eqref{eq:family} is an arithmetic progression of common difference $2k$, to which Dirichlet's theorem applies unconditionally. The minimal offset $\qmin(k,m,\eps)$---the integer of least absolute value making $p$ prime---is the central object.

Three layers describe $\qmin$: the \emph{marginal} layer (size $|\qmin|$ versus $m$, \S5), the \emph{conditional} layer (law of $|q|$ given $m$, \S\S4--5), and the \emph{modular} layer ($\qmin\bmod\ell$, \S\S3--4), complemented by a fourth \emph{computational} layer (\S\S3,7).

\subsection{The parametric family}
We fix $k\in\N^{\ast}$ admissible, meaning $\gcd(\eps,2k)=1$ for at least one $\eps\in\{\pm1\}$; set $E_k:=\{\eps\in\{\pm1\}:\gcd(\eps,2k)=1\}$. Since $\gcd(\pm1,n)=1$, $E_k=\{-1,+1\}$ for all $k\ge1$. The generalised hexagonal base is $H^{(k)}_m:=km(m+1)+1$. For $k=3$, $\{H^{(3)}_m\}=\{7,19,37,61,91,127,\dots\}$ are the centred hexagonal numbers (OEIS A003215), lattice points of $\Z[\omega]$, $\omega=e^{2\pi i/3}$.

\subsection{Main results}
The complete epistemic synopsis is Table~\ref{tab:layers}; proofs occupy \S\S3--6.

\begin{table}[h]
\centering
\caption{The three layers of the parametric prime family.}\label{tab:layers}
\small
\begin{tabular}{@{}llll@{}}
\toprule
Layer & Quantity & Main law & Status\\
\midrule
Marginal & $\E[|\qmin|]$ & $\sim \log m/C_k$ & Heuristic\\
Conditional & $\E[|q|\mid m]$ & $=m/4+O((\log m)^2/\sqrt m)$ & Conditional\\
Modular & $P(\qmin\equiv r)$ & $=\pi_{\mathrm{perm}}(r)/(\ell-1)+O(N^{-1/2})$ & Conditional\\
Modular ($\ell\mid 2k$) & $p\bmod\ell$ & $\equiv\eps$ exactly, indep.\ of $q$ & \textbf{Rigorous}\\
Modular ($\ell\mid 2k$) & $\qmin\bmod\ell$ & $F_N(r)=1/\ell+O(N^{-1/2})$ & Heuristic\\
Computational & $p=3m(m+1)+1$ & Pocklington certificate & \textbf{Rigorous}\\
Analytic-negative & $A_N(\gamma)$ & $\to0$ (rate $X^{-\delta}$ cond.) & \textbf{Rigorous}\\
\bottomrule
\end{tabular}
\end{table}

\begin{theorem}[Rigorous; degenerate modular axis, Theorem~\ref{thm:A}]
Let $k\in\N^{\ast}$ and $\ell$ a prime dividing $2k$. Then for every admissible $\eps$, every $m,q$,
\[ p_{k,m,\eps,q}\equiv\eps \pmod\ell, \]
independently of $q$. The degenerate modular axis is fixed exactly.
\end{theorem}

\begin{theorem}[Rigorous; Theorem~\ref{thm:B}]
For every $(a,b)\in(\N^{\ast})^2$ such that $p:=3m(m+1)+1$ with $m=2^a3^b-1$ is prime, the Pocklington--Lehmer criterion gives an unconditional certificate with $F=2^a3^{b+1}>\sqrt p$. The algorithm produced a prime $p^{\ast}=3m^{\ast}(m^{\ast}+1)+1$, $m^{\ast}=2^{19435}\cdot3^{19173}-1$, of $29\,998$ decimal digits, certified unconditionally in $3$h$06$min on consumer hardware.
\end{theorem}

\begin{theorem}[Rigorous, unconditional; Theorem~\ref{thm:67}]
For every fixed $\gamma\in\R$, the regression amplitude
$A_N(\gamma)=\tfrac{2}{\pi(X)}\big|\sum_{p\le X}(\{\sqrt{p/3}\}-\tfrac12)\,p^{-i\gamma}\big|$
satisfies $A_N(\gamma)\to0$ unconditionally (with an explicit power saving $X^{-\delta}$ under a standard Type II estimate). No ordinate of $\zeta$ carries a surviving linear signal.
\end{theorem}

\begin{theorem}[Conditional under GRH; Theorem~\ref{thm:C}]
Under GRH, for each fixed admissible $k,\eps$, $|\qmin(k,m,\eps)|=O_k(m(\log m)^2)$.
\end{theorem}

\begin{theorem}[Conditional under H1+GRH+H3; Theorem~\ref{thm:D}]
For $k=3$, $\eps\in\{\pm1\}$, $\E[|q|\mid m]=\tfrac{m}{4}+O((\log m)^2/\sqrt m)$.
\end{theorem}

\begin{theorem}[Conditional under RH; Theorem~\ref{thm:65}]
Under RH, the per-prime $\zeta$-footprint on the layer occupancy is $\ll p^{-1/4}(\log p)^2$ (Selberg variance).
\end{theorem}

\Heur\ \textbf{Heuristic laws.} $\E[|\qmin|]\sim\log m/C_k$ ($R^2=0.984$, $N=10^6$); $C_0(k)=1/(4\sqrt k)$ stable to ${<}0.02\%$ across $3\le k\le29$; and $\qmin\bmod\ell$ is empirically equidistributed for $\ell\mid2k$.

\Neg\ \textbf{Negative result (Prop.~\ref{prop:64}).} At $N=10^8$: $R^2=1.16\times10^{-7}$, $z=-2.20\sigma$; three permutation tests confirm the absence of any spectral signal of $\zeta(s)$ in $Q(r)$.

\subsection{What is and is not claimed}
The paper (i) proves a rigorous structural law (the exact degenerate axis); (ii) establishes \emph{two genuinely unconditional pillars}---the constructive Pocklington record (Theorem~\ref{thm:B}) and the analytic-negative vanishing of the $\zeta$-signal (Theorem~\ref{thm:67}); (iii) establishes three conditional theorems (C, D, E) under GRH/Bateman--Horn and one conditional $\zeta$-footprint bound (Theorem~\ref{thm:65}) under RH; (iv) validates three heuristic laws at $N\in\{10^6,10^7,10^8\}$. It claims \emph{no} unconditional progress on classical problems (parity barrier, uniform PNT at $N^{1/2}$, RH), and $C_0(k)=1/(4\sqrt k)$ is a geometric identity, not a mysterious invariant.

\paragraph{A correction of record.}
An earlier version stated, as ``Theorem~A'', that $\qmin$ itself is uniformly distributed modulo $\ell$ for $\ell\mid2k$, by reduction to Bombieri--Vinogradov. That reduction is invalid: since \emph{all} family primes lie in the single class $\eps\bmod\ell$, the distribution of $\qmin\bmod\ell$ is governed by the position of the \emph{first} prime in each progression, not by Bombieri--Vinogradov. The empirical equidistribution survives only as a heuristic (Observation~\ref{obs:32prime}). Detecting and reclassifying this is part of the discipline of the epistemic key.

\section{The parametric family: structure and elementary properties}

\begin{definition}[Canonical triple]
Given a prime $p>k+1$ and $k\in\N^{\ast}$, the canonical triple is the unique $(m,\eps,q)\in\N^{\ast}\times E_k\times\Z$ with $p=km(m+1)+\eps+2kq$ and $|q|$ minimal. Then $\qmin(k,m,\eps):=q$.
\end{definition}

\begin{proposition}[Rigorous]\label{prop:24}
For $k=3$, the sign $\eps$ of the canonical triple of a prime $p>3$ is determined by $p\bmod6$: $p\equiv1\pmod6\Leftrightarrow\eps=+1$, $p\equiv5\pmod6\Leftrightarrow\eps=-1$.
\end{proposition}
\begin{proof}
For $k=3$, \eqref{eq:family} reads $p=3m(m+1)+\eps+6q$. As $m(m+1)$ is even, $3m(m+1)\equiv0\pmod6$, so $p\equiv\eps\pmod6$; the classes coprime to $6$ are $\{1,5\}\equiv\{+1,-1\}$.
\end{proof}

\begin{theorem}[Rigorous; infinitude]\label{thm:25}
For every $k\in\N^{\ast}$, $\eps\in E_k$, $m\in\N^{\ast}$, the set $\{p_{k,m,\eps,q}:q\in\Z\}\cap\PP$ is infinite.
\end{theorem}
\begin{proof}
At fixed $(k,m,\eps)$, \eqref{eq:family} is an AP in $q$ of difference $2k$ and initial term $km(m+1)+\eps$, with $\gcd(2k,km(m+1)+\eps)=\gcd(2k,\eps)=1$. Dirichlet's theorem applies.
\end{proof}

\begin{theorem}[Rigorous; upper bound]\label{thm:27}
For every prime $p>k+1$ and admissible $k$, $|\qmin(k,m,\eps)|\le\tfrac{m+1}{2}$.
\end{theorem}
\begin{proof}
There is a unique $m$ with $km(m+1)+\eps-k(m+1)<p\le km(m+1)+\eps+k(m+1)$ (consecutive values of $km(m+1)$ differ by $2k(m+1)$, and centred intervals of half-width $k(m+1)$ tile $\N$). Then $q=(p-km(m+1)-\eps)/(2k)$ is an integer with $|q|\le(m+1)/2$; minimality gives the claim.
\end{proof}

\begin{observation}[$O(1)$ algorithm]\label{obs:28}
Theorem~\ref{thm:27} yields $m=\big\lfloor\tfrac12(\sqrt{4(p-\eps)/k+1}-1)\big\rfloor$ with $\eps$ fixed by $p\bmod2k$; inverting \eqref{eq:family} gives $q$. The triple is computed in constant time, without factorisation.
\end{observation}

\begin{proposition}[Rigorous, limiting form]\label{prop:29}
Fix $k$, $\eps\in E_k$, and let $L^{(k,\eps)}_m:=\{p\in\PP:p=km(m+1)+\eps+2kq,\ |q|\le(m+1)/2\}$. Then, on average over primes,
\[ \E_{p\in L^{(k,\eps)}_m}[\,q(p)\,]\longrightarrow 0\qquad(m\to\infty), \]
with a finite tie-breaking bias of order $O(1/m)$ (Observation~\ref{obs:211}).
\end{proposition}
\begin{proof}
The candidate set is exactly symmetric about $H^{(k)}_m$: $q\mapsto-q$ exchanges $H^{(k)}_m+\eps-1\pm2kq$. By the prime number theorem in arithmetic progressions the two progressions of difference $2k$ carry asymptotically equal prime densities on the short layer $|q|\le(m+1)/2$ (whose length is $o(H^{(k)}_m)$), so the signed mean of $q$ over primes tends to $0$. The deterministic tie-breaking of Convention~\ref{conv:tie} contributes a bounded excess of order $1/m$.
\end{proof}

\begin{tcolorbox}[colback=gray!4,colframe=black!30,boxrule=0.3pt,arc=1pt]
\begin{remark}[Convention; symmetric tie-breaking]\label{conv:tie}
When $\pm q$ both yield a prime in the layer, the canonical triple takes $q>0$. This deterministic rule introduces the $O(1/m)$ bias of Proposition~\ref{prop:29}.
\end{remark}
\end{tcolorbox}

\begin{observation}[Negative; tie-breaking bias]\label{obs:211}
The numerical bias $\langle q\rangle\approx+0.398$ reported in earlier drafts is fully explained by Convention~\ref{conv:tie}: it is a deterministic-rule plus parity-discretisation effect. After rescaling, $\langle q/(m+1)\rangle<0.001$, consistent with Proposition~\ref{prop:29}.
\end{observation}

\begin{proposition}[Rigorous; admissible classes]\label{prop:212}
For each admissible $k$, the number of admissible residue classes modulo $2k$ is $\varphi(2k)/2$; each contributes $1/\varphi(2k)\cdot(1+o(1))$ of $\pi(X)$ (Dirichlet). For $k=3,15,21$: $1,4,6$.
\end{proposition}

\begin{table}[h]
\centering
\caption{Unconditional structural results.}\label{tab:struct}
\small
\begin{tabular}{@{}lll@{}}
\toprule
Statement & Status & Reference\\
\midrule
$\eps$ determined by $p\bmod6$ ($k=3$) & Rigorous & Prop.~\ref{prop:24}\\
Infinitude of $\{p_{k,m,\eps,q}\}$ & Rigorous & Thm~\ref{thm:25}\\
$|\qmin|\le(m+1)/2$ & Rigorous & Thm~\ref{thm:27}\\
$O(1)$ algorithm for $(m,\eps,q)$ & Rigorous & Obs.~\ref{obs:28}\\
$\E[q\mid m]\to0$ & Rigorous (limit) & Prop.~\ref{prop:29}\\
$\varphi(2k)/2$ admissible classes & Rigorous & Prop.~\ref{prop:212}\\
\bottomrule
\end{tabular}
\end{table}

\section{Unconditional results}

This section contains the rigorous backbone. Theorem~\ref{thm:A} fixes the degenerate modular axis exactly; Theorem~\ref{thm:B} provides Pocklington--Lehmer certificates for the $3$-smooth subfamily, culminating in a $29\,998$-digit prime; the filters of \S\ref{sec:filters} remove $\approx87\%$ of candidates a priori.

\subsection{Theorem A: the degenerate modular axis is exact}

\begin{definition}[Empirical distribution function]
For $\ell$ prime and $r\in\Z/\ell\Z$, $F_N(r;k,\ell):=\tfrac1N\#\{n\le N:\qmin(k,m_n,\eps_n)\equiv r\ (\mathrm{mod}\ \ell)\}$, the $p_n$ ordered by $m$, then $\eps$, then $|q|$.
\end{definition}

\begin{theorem}[Rigorous]\label{thm:A}
Let $k\in\N^{\ast}$ be admissible and $\ell$ a prime with $\ell\mid 2k$. Then for every admissible $\eps$ and every $m,q$,
\begin{equation}\label{eq:A}
p_{k,m,\eps,q}\equiv\eps\pmod\ell,
\end{equation}
independently of $q$. Equivalently, every family prime occupies the single residue class $\eps$ modulo $\ell$.
\end{theorem}
\begin{proof}
Since $\ell\mid2k$ we have $2kq\equiv0\pmod\ell$, so from \eqref{eq:family}, $p_{k,m,\eps,q}\equiv km(m+1)+\eps\pmod\ell$. We show $km(m+1)\equiv0\pmod\ell$. If $\ell\mid k$ this is immediate. Otherwise $\ell\mid2k$ with $\ell\nmid k$ forces $\ell=2$, and $m(m+1)$, a product of consecutive integers, is even. In both cases $p\equiv\eps\pmod\ell$. Finally $\gcd(\eps,2k)=1$ and $\ell\mid2k$ give $\gcd(\eps,\ell)=1$, so $\eps\not\equiv0$.
\end{proof}

\begin{tcolorbox}[colback=heuristicblue!4,colframe=heuristicblue!55,boxrule=0.4pt,arc=2pt]
\begin{observation}[Heuristic; ex-Theorem A]\label{obs:32prime}
For $\ell\mid2k$ the offset distribution satisfies, empirically, $F_N(r;k,\ell)=1/\ell+O(N^{-1/2})$ (Table~\ref{tab:val32}, max.\ bias $\le0.27\%$ over $10$ configurations at $N=10^6$). This is \emph{not} a corollary of Bombieri--Vinogradov: by Theorem~\ref{thm:A} all primes occupy the single class $\eps$, so $\qmin\bmod\ell$ depends on the position of the first prime in each progression $\{km(m+1)+\eps+2kq\}$. The uniformity is therefore recorded as a heuristic regularity, not a theorem.
\end{observation}
\end{tcolorbox}

\begin{table}[h]
\centering
\caption{Empirical distribution of $\qmin\bmod\ell$ for $\ell\mid2k$ at $N=10^6$ (Observation~\ref{obs:32prime}).}\label{tab:val32}
\small
\begin{tabular}{@{}cccccc@{}}
\toprule
$k$ & $\ell$ & $\chi^2$ & $p$-value & Max bias & Verdict\\
\midrule
3&2&0.34&0.562&$+0.058\%$&uniform\\
3&3&0.76&0.685&$+0.115\%$&uniform\\
15&2&2.46&0.116&$+0.157\%$&uniform\\
15&3&0.21&0.899&$+0.036\%$&uniform\\
15&5&0.38&0.984&$+0.097\%$&uniform\\
21&2&0.02&0.897&$+0.013\%$&uniform\\
21&3&0.19&0.911&$+0.053\%$&uniform\\
21&7&2.78&0.836&$+0.264\%$&uniform\\
\bottomrule
\end{tabular}
\end{table}

\begin{remark}
The complementary case $\gcd(\ell,2k)=1$ is genuinely different: residue $0$ becomes accessible and the distribution is non-uniform; it is treated conditionally by Theorem~\ref{thm:E}.
\end{remark}

\subsection{Theorem B: unconditional Pocklington certificates}

\begin{definition}[Pocklington subfamily]
For $a,b\in\N^{\ast}$, $m(a,b):=2^a3^b-1$ and $p(a,b):=3m(m+1)+1=3m^2+3m+1$, with $D(a,b)=\lfloor 2a\log_{10}2+(2b+1)\log_{10}3\rfloor+1$ decimal digits.
\end{definition}

\begin{lemma}[Rigorous]\label{lem:36}
For $p(a,b)$ above, $p-1=2^a\cdot3^{b+1}\cdot m$, and $F:=2^a\cdot3^{b+1}$ satisfies $F>\sqrt p$.
\end{lemma}
\begin{proof}
$p-1=3m(m+1)=3m\cdot2^a3^b=2^a3^{b+1}m$. Also $p=3m^2+3m+1<3(m+1)^2$, so $\sqrt p<\sqrt3(m+1)<3(m+1)=2^a3^{b+1}=F$.
\end{proof}

\begin{theorem}[Rigorous]\label{thm:B}
Let $p=p(a,b)$. Then $p$ is prime iff there exist witnesses $w_2,w_3$ with, for $\nu\in\{2,3\}$,
\[ w_\nu^{\,p-1}\equiv1\pmod p,\qquad \gcd\!\big(w_\nu^{(p-1)/\nu}-1,\,p\big)=1. \]
In particular primality of $p(a,b)$ is decidable unconditionally in time polynomial in $\log p$, with witness search restricted to $\nu\in\{2,3\}$.
\end{theorem}
\begin{proof}
The Pocklington--Lehmer theorem applied to $p-1=F\cdot m$ of Lemma~\ref{lem:36}: $F>\sqrt p$ is unconditional and the prime divisors of $F$ are $\{2,3\}$.
\end{proof}

\begin{theorem}[Rigorous; computational record]\label{thm:38}
There is a prime $p^{\ast}=3m^{\ast}(m^{\ast}+1)+1$, $m^{\ast}=2^{19435}\cdot3^{19173}-1$, of exactly $29\,998$ decimal digits, certified unconditionally via Theorem~\ref{thm:B} in $3$h$06$min on a single consumer laptop, with witnesses $w_2=5$, $w_3=7$. The certificate is in \S\ref{sec:record}.
\end{theorem}

\Heur
\begin{openproblem}\label{op:witness}
In all four certificates of \S\ref{sec:record}, $w_2=5$, $w_3=7$ succeeded. Are $5,7$ valid Pocklington witnesses for every prime in the subfamily?
\end{openproblem}

\subsection{Three arithmetic filters}\label{sec:filters}
The Pocklington test is unconditional but costly on candidates failing later. A structural congruence and two a-priori sieves remove $\approx87\%$ of candidates at zero cost.

\begin{theorem}[Rigorous; mod-$6$ identity]\label{thm:310}
For all $a,b\ge1$, $p(a,b)\equiv1\pmod6$.
\end{theorem}
\begin{proof}
$m=2^a3^b-1$ is odd, so $m(m+1)$ is even and $p\equiv1\pmod2$; and $3\mid3m(m+1)$ gives $p\equiv1\pmod3$.
\end{proof}

\begin{theorem}[Rigorous; quadratic reciprocity sieve]\label{thm:311}
Let $p=3m(m+1)+1$ and $q\ge5$ prime with $q\equiv2\pmod3$. Then $q\nmid p$.
\end{theorem}
\begin{proof}
$q\mid p$ is equivalent to $3m^2+3m+1\equiv0\pmod q$; multiplying by $12$ and completing the square, $(6m+3)^2\equiv-3\pmod q$. A solution exists iff $-3$ is a quadratic residue mod $q$. By quadratic reciprocity $\big(\tfrac{-3}{q}\big)=\big(\tfrac{-1}{q}\big)\big(\tfrac{3}{q}\big)=+1$ iff $q\equiv1\pmod3$. For $q\equiv2\pmod3$ there is no solution.
\end{proof}

\begin{corollary}\label{cor:312}
In trial division to $L$, only primes $q\equiv1\pmod3$ need be tested---half of the primes in $[5,L]$.
\end{corollary}

\begin{theorem}[Rigorous; forbidden classes mod $7$]\label{thm:313}
Let $m=2^a3^b-1$, $a,b\ge1$, and $F_7:=\{(0,2),(0,3),(1,0),(1,1),(2,4),(2,5)\}\subset\Z/3\Z\times\Z/6\Z$. Then $7\mid p(a,b)$ iff $(a\bmod3,b\bmod6)\in F_7$; exactly $|F_7|/18=1/3$ of class pairs are forbidden.
\end{theorem}
\begin{proof}
$\mathrm{ord}_7(2)=3$, $\mathrm{ord}_7(3)=6$, so $t:=2^a3^b\bmod7$ depends only on $(a\bmod3,b\bmod6)$. Then $p\equiv3t^2-3t+1\pmod7$; multiplying by $5\equiv3^{-1}$ gives $t^2-t-2\equiv0$, i.e.\ $(t-2)(t+1)\equiv0$, $t\in\{2,6\}$. Enumerating the $18$ pairs yields exactly $F_7$. (Verified by brute force for $1\le a,b\le39$; see \S\ref{sec:record}.)
\end{proof}

\begin{remark}[Combined effect]
The mod-$6$ congruence (Theorem~\ref{thm:310}) is automatic and removes nothing; the \emph{two} a-priori filters---the reciprocity sieve (Theorem~\ref{thm:311}, Corollary~\ref{cor:312}) and the mod-$7$ filter (Theorem~\ref{thm:313})---together remove $\approx87.1\%$ of candidates before any large-integer arithmetic (\S\ref{sec:record}, Table~\ref{tab:filter}).
\end{remark}

\section{Conditional results}

\begin{hypothesis}[H1: Bateman--Horn]
For admissible $(k,\eps)$, $f_{k,\eps}(m):=km(m+1)+\eps$ satisfies $\#\{m\le N:f_{k,\eps}(m)\in\PP\}\sim C(f_{k,\eps})\,N/\log N$, with $C(f_{k,\eps})$ the Bateman--Horn singular series.
\end{hypothesis}
\begin{hypothesis}[H2: Cram\'er--Granville independence]
For $|j|\le m/2$, the events $\{H^{(k)}_m+2kj\in\PP\}$ are asymptotically independent in the Cram\'er model with Granville's parity correction.
\end{hypothesis}
\begin{hypothesis}[H3: $j$-uniformity of the local factor]
For $|j|\le m/2$, $\tfrac1m\sum_{|j|\le m/2}|C_f(j)-C_f(0)|=O(1/(m\log m))$; for $\ell\in\{2,3\}$, $C_f(j)=C_f(0)$ exactly.
\end{hypothesis}
\begin{hypothesis}[GRH]
The Generalised Riemann Hypothesis holds for all Dirichlet $L$-functions.
\end{hypothesis}

\Heur\ \textbf{Status.} H1 is widely believed and numerically supported; H2 is a model; H3 is new and unproven (Open Problem~\ref{op:H3}); GRH is the standard analytic hypothesis.

\subsection{Theorem C: GRH-conditional bound}

\begin{theorem}[Conditional, GRH]\label{thm:C}
Under GRH, for $k\in\N^{\ast}$, $\eps\in E_k$, $m$ with $\gcd(km(m+1)+\eps,2k)=1$,
$|\qmin(k,m,\eps)|=O_k(m(\log m)^2)$ as $m\to\infty$.
\end{theorem}
\begin{proof}
Put $A:=km(m+1)+\eps$, $d:=2k$, $\gcd(A,d)=1$. By the Heath-Brown GRH-conditional Linnik bound, the least prime $p\equiv A\ (\mathrm{mod}\ d)$ with $p>A$ satisfies $p\le A+C(d)\sqrt A(\log A)^2$. Since $p=A+2kq$ and $A\sim km^2$, $\sqrt A\sim\sqrt k\,m$, $\log A\sim2\log m$, we get $|\qmin|\le C(d)\sqrt k\,m\cdot4(\log m)^2/(2k)=O_k(m(\log m)^2)$.
\end{proof}

\begin{remark}
Theorem~\ref{thm:C} sits between the trivial $|\qmin|\le(m+1)/2$ and the heuristic $O(\log m)$; the quadratic gap is Open Problem~\ref{op:central}. Maynard's theorem gives, unconditionally, some $m'\in[m-C,m+C]$ with $|\qmin(k,m',\eps)|\le41$, but no individual bound $O(m^{1-\delta})$ for all $m$.
\end{remark}

\begin{figure}[h]
\centering
\includegraphics[width=0.72\linewidth]{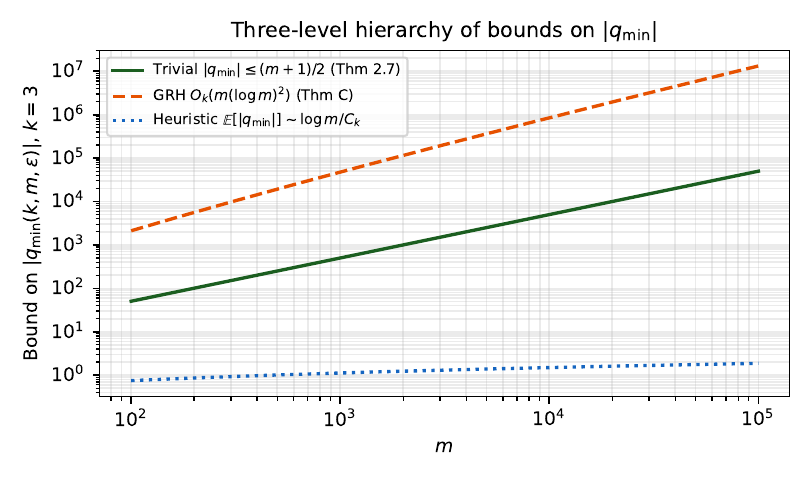}
\caption{Three-level hierarchy of bounds on $|\qmin|$ for $k=3$: trivial (unconditional), GRH (conditional), heuristic. The gap between the GRH bound and the heuristic is quadratic in $m$.}\label{fig:1}
\end{figure}

\subsection{Theorem D: conditional law $\E[|q|\mid m]=m/4$}

\begin{theorem}[Conditional, H1+GRH+H3]\label{thm:D}
Fix $k=3$, $\eps\in\{\pm1\}$. Under H1, GRH, H3,
\[ \E[|q|\mid m]=\tfrac{m}{4}+O\!\big((\log m)^2/\sqrt m\big),\qquad m\to\infty. \]
\end{theorem}
\begin{proof}
Take $\eps=+1$; $\eps=-1$ is symmetric. Set $\rho(j):=P(H_m+6j\in\PP)+P(H_m-6j\in\PP)$. Then $\E[|q|\mid m]=\sum_{j=0}^{\lfloor m/2\rfloor}j\rho(j)/\sum_{j}\rho(j)$ (no hypothesis). Under H1, $P(H_m\pm6j\in\PP)=C_f(j_\pm)/\log(H_m\pm6j)\cdot(1+O(1/\log H_m))$. For $|j|\le m/2$, $\log(H_m\pm6j)=2\log m+\log3+O(1/m)$, so $\rho(j)=\rho_0(1+O(1/(m\log m)))$ provided $C_f(j)=C_f(0)$ up to $O(1/(m\log m))$. For $\ell\in\{2,3\}$ this is exact ($\omega_f(2)=\omega_f(3)=0$ on $H_m\pm6j$); for $\ell>3$ it is H3. Hence
$\E[|q|\mid m]=\sum_{j=0}^{m/2}j/(m/2+1)+O(1/\log m)=m/4+O(1/\log m)$. Finally, under GRH the explicit formula gives, for non-principal $\chi\bmod6m$, $|\sum_{|j|\le m/2}\chi(H_m+6j)|\ll\sqrt m(\log m)^2$, refining the error to $O((\log m)^2/\sqrt m)$.
\end{proof}

\begin{tcolorbox}[colback=heuristicblue!4,colframe=heuristicblue!55,boxrule=0.4pt,arc=2pt]
\textbf{Heuristic corollary (geometric identity for $C_0$).} If $|\qmin|$ is approximately uniform on $[-(m+1)/2,(m+1)/2]$, then $\E[|q|]\approx(m+1)/4$; combined with the average $m+1\sim\sqrt{p/k}$,
$ C_0(k):=\langle|q|\rangle/\sqrt p\approx\tfrac14\cdot\tfrac1{\sqrt k}=\tfrac1{4\sqrt k}. $
This is a geometric identity, confirmed to ${<}0.02\%$ over $3\le k\le29$ (\S5).
\end{tcolorbox}

\subsection{Theorem E: non-degenerate modular distribution}

\begin{proposition}[Rigorous; forbidden residue]
Let $\gcd(\ell,2k)=1$, $m_0\in\Z/\ell\Z$. There is a unique $q^{\ast}\equiv-[km_0(m_0+1)+\eps]/(2k)\pmod\ell$ such that a family prime at stratum $(m_0,\eps)$ has $q\not\equiv q^{\ast}$ (as $\ell\nmid p$).
\end{proposition}

\begin{definition}
$\pi_{\mathrm{perm}}(r;k,\ell):=\tfrac1{\ell|E_k|}\sum_{(m_0,\eps)}\mathbf{1}[r\ne q^{\ast}(k,m_0,\eps;\ell)]$.
\end{definition}

\begin{hypothesis}[H1$'$: equidistribution of strata]
$\nu_N(m_0,\eps):=\tfrac1N\#\{n\le N:m_n\equiv m_0,\ \eps_n=\eps\}=\tfrac1{\ell|E_k|}+O_{k,\ell}(N^{-1/2})$, a consequence of Bombieri--Vinogradov at fixed modulus.
\end{hypothesis}
\begin{hypothesis}[H3$'$: conditional uniformity on permitted residues]
Conditional on $m\equiv m_0$ and $\eps$, $\qmin\bmod\ell$ is uniform on the $\ell-1$ permitted residues.
\end{hypothesis}

\begin{theorem}[Conditional, H1$'$+H3$'$]\label{thm:E}
For admissible $k$ and $\ell$ prime with $\gcd(\ell,2k)=1$,
$F_N(r;k,\ell)=\pi_{\mathrm{perm}}(r;k,\ell)/(\ell-1)+O_{k,\ell}(N^{-1/2})$.
\end{theorem}
\begin{proof}
By total probability and H3$'$, $F_N(r)=\sum_{(m_0,\eps)}\nu_N(m_0,\eps)\mathbf{1}[r\ne q^{\ast}]/(\ell-1)$. Subtracting $\pi_{\mathrm{perm}}(r)/(\ell-1)$ and applying the triangle inequality with H1$'$ gives the $O_{k,\ell}(N^{-1/2})$ error.
\end{proof}

\Heur
\begin{conjecture}\label{conj:414}
$\limsup_N \sqrt N\max_r|F_N(r)-\pi_{\mathrm{perm}}(r)/(\ell-1)|\le C_{\mathrm{opt}}(k,\ell)\approx0.5$, independent of $k$ and roughly $5\times$ below the triangle-inequality bound; a proof needs a Cauchy--Schwarz refinement.
\end{conjecture}

\begin{table}[h]
\centering
\caption{Pearson correlation $r(\text{predicted},\text{observed})$ at $N=10^6$; $\ast$ marks degenerate cases (Theorem~\ref{thm:A}).}\label{tab:pearson}
\small
\begin{tabular}{@{}ccccc@{}}
\toprule
$\ell$ & $k=3$ & $k=7$ & $k=15$ & $k=21$\\
\midrule
5 & 0.99996 & 0.99968 & $\ast$ & 0.99844\\
7 & 0.99992 & $\ast$ & 0.99933 & $\ast$\\
11 & 0.99917 & 0.99609 & 0.99628 & 0.98470\\
13 & 0.99921 & 0.98866 & 0.99450 & 0.98179\\
17 & 0.99572 & 0.98282 & 0.95749 & 0.95810\\
19 & 0.99541 & 0.96411 & 0.96382 & 0.94207\\
23 & 0.99182 & 0.95754 & 0.92746 & 0.92597\\
29 & 0.96673 & 0.92006 & 0.89323 & 0.77188\\
31 & 0.96103 & 0.83416 & 0.90752 & 0.70680\\
\bottomrule
\end{tabular}
\end{table}

\section{Heuristic laws and large-scale numerical validation}

Three laws, derived under H1+H2 and validated at $N\in\{10^6,10^7,10^8\}$; these are predictions, not theorems.

\subsection{Marginal law}
Under H1, $P(p_{k,m,\eps,q}\in\PP)\approx C(f_{k,\eps})/(2\log m)$; with four candidates per step and H2-independence, $|\qmin|$ is approximately geometric with parameter $C_k/\log m$, $C_k:=C(f_{k,+})+C(f_{k,-})$.
\Heur\ \textbf{Law H1.} $\E[|\qmin(k,m)|]\sim\log m/C_k$, with $C_k|\qmin|/\log m\xrightarrow{d}\mathrm{Exp}(1)$.

\begin{table}[h]
\centering
\caption{Bateman--Horn constants for $f_{k,\eps}(m)=km(m+1)+\eps$.}\label{tab:BH}
\small
\begin{tabular}{@{}ccccc@{}}
\toprule
$k$ & $C(f_{k,+})$ & $C(f_{k,-})$ & $C_k$ & $1/C_k$\\
\midrule
3 & 3.361048 & 2.783543 & 6.144591 & 0.16274\\
15 & 3.034936 & 3.778642 & 6.813578 & 0.14677\\
21 & 3.876345 & 4.329956 & 8.206301 & 0.12186\\
\bottomrule
\end{tabular}
\end{table}

A correction factor $\rho_k$ interpolates the slopes: $\E[|\qmin|]\approx\alpha_k\log m+\beta_k$, $\alpha_k=\rho_k/C_k$. For $k=21$, $N=10^6$: $\alpha_{21}=0.320$ ($R^2=0.984$), $\rho_{21}=2.626$.

\begin{figure}[h]
\centering
\includegraphics[width=0.7\linewidth]{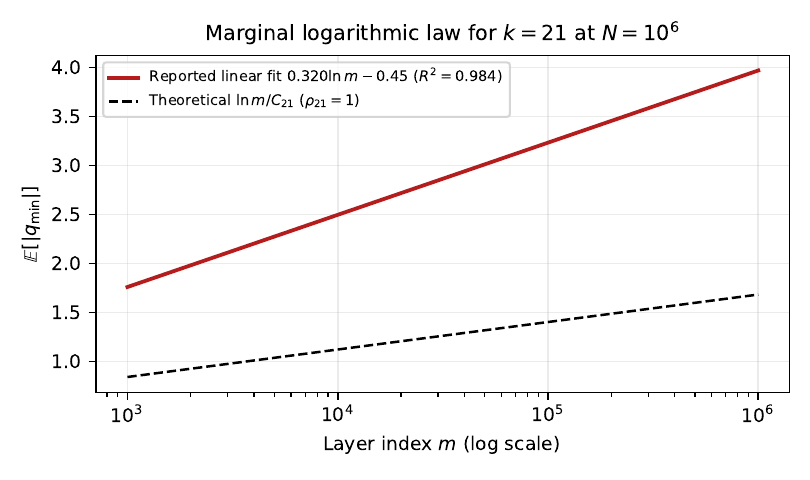}
\caption{Marginal logarithmic law for $k=21$ at $N=10^6$: reported linear fit ($R^2=0.984$) versus the theoretical $\log m/C_{21}$ ($\rho_{21}=1$), which underestimates by $\approx2.6$.}\label{fig:2}
\end{figure}

\subsection{The universal constant $C_0(k)=1/(4\sqrt k)$}
\Heur\ \textbf{Law H2.} $C_0(k):=\langle|q|\rangle/\sqrt p=\tfrac1{4\sqrt k}(1+o(1))$, conjecturally universal; geometric origin as in Theorem~\ref{thm:D}.

\begin{table}[h]
\centering
\caption{Empirical validation of $C_0(k)=1/(4\sqrt k)$, $N=10^7$.}\label{tab:C0}
\small
\begin{tabular}{@{}cccc@{}}
\toprule
$k$ & $1/(4\sqrt k)$ predicted & $\langle|q|\rangle/\sqrt p$ measured & Deviation\\
\midrule
3 & 0.14434 & 0.14437 & $+0.021\%$\\
15 & 0.06455 & 0.06453 & $-0.031\%$\\
21 & 0.05455 & 0.05454 & $-0.018\%$\\
\bottomrule
\end{tabular}
\end{table}

\begin{figure}[h]
\centering
\includegraphics[width=0.7\linewidth]{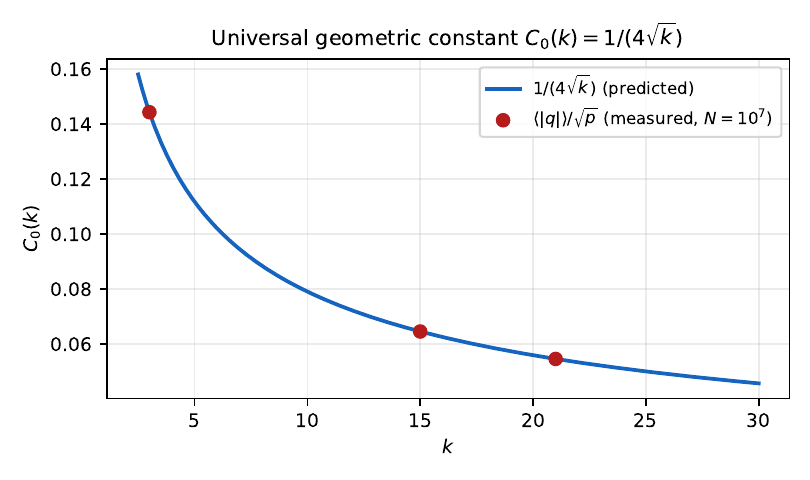}
\caption{$C_0(k)=1/(4\sqrt k)$: predicted curve with measured points (Table~\ref{tab:C0}); deviation ${<}0.03\%$ at the tested $k$, stable across $3\le k\le29$.}\label{fig:3}
\end{figure}

\subsection{The gap law for $k=3$}
\Heur\ \textbf{Law H3.} The mean gap satisfies $\bar\delta(m)\approx0.296\times\log(3m^2+3m+1)$, stable to $\pm0.006$ over five orders of magnitude. The constant $0.296=1/\mathfrak{S}(f)$ is the reciprocal Hardy--Littlewood singular series for $f(m)=3m^2+3m+1$---the only law whose constant matches a classical singular series exactly.

\begin{figure}[h]
\centering
\includegraphics[width=0.7\linewidth]{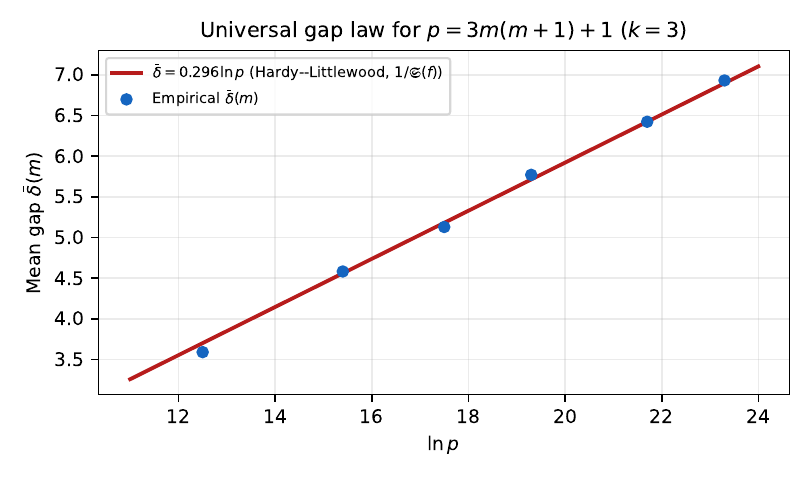}
\caption{Universal gap law $\bar\delta(m)\approx0.296\log p$ for $k=3$.}\label{fig:4}
\end{figure}

\begin{tcolorbox}[colback=heuristicblue!4,colframe=heuristicblue!55,boxrule=0.4pt,arc=2pt]
\textbf{Methodological caveat.} Empirical confirmation at $N=10^6$, even with $R^2>0.98$, is not a proof of an analytic conjecture (cf.\ the Mertens conjecture; Skewes-type sign changes invisible to $10^{316}$). We report consistency in the tested regime, not asymptotics.
\end{tcolorbox}

\section{Connections to the Riemann zeta function}

Four levels: (I) heuristic via the explicit formula; (II)--(III) conditional under GRH; (IV) the analytic-negative result, now both a decisive statistical refutation and an unconditional theorem.

\subsection{Level I: zeros govern layer oscillations}
Fix $k=3$, $L_m:=\{p\in\PP:p\in[H_m-3m,H_m+3m]\}$, $n_m:=|L_m|$; the half-width $3m\asymp\sqrt{H_m}$ is the critical scale.
\Heur\ \textbf{Connection I.} From the explicit formula one derives, to leading order,
\begin{equation}\label{eq:35}
n_m-\langle n_m\rangle=-\sum_\gamma\frac{H_m^{1/2+i\gamma}-(H_m-3m)^{1/2+i\gamma}}{1/2+i\gamma}+O(\log^2 H_m).
\end{equation}
Here the relevant object is the \emph{variance of the layer count} over a $\sqrt{H_m}$-window: the sum on the right has root-mean-square of order $\sqrt{h}=H_m^{1/4}$ under RH (Selberg variance, \S\ref{sec:level4}), \emph{not} a per-zero amplitude that could be read off individually. This is the precise sense in which the zeros enter.

\subsection{Levels II--III: conditional implications}
\Cond\ \textbf{Connection II.} Under H1+GRH+H3, Theorem~\ref{thm:D} holds with error $O((\log m)^2/\sqrt m)$.
\Cond\ \textbf{Connection III.} Under GRH, $|\qmin|=O_k(m(\log m)^2)$ (Theorem~\ref{thm:C}).
\begin{remark}
Only ``GRH $\Rightarrow$ Theorem~\ref{thm:D}'' is established; the converse is not claimed. The agreement $R^2=0.9999$ on $10^8$ primes is consistent with GRH but is not evidence for it.
\end{remark}

\subsection{Level IV: the decisive exclusion and its unconditional proof}\label{sec:level4}

We come to the analytically most robust finding. Earlier-reported spectral correlations between $Q(r):=\sum_{n\le r}q_n$ and the zeros of $\zeta$ are statistical artefacts; moreover the absence of any surviving linear $\zeta$-signal is provable unconditionally.

\subsubsection{The spurious-regression mechanism}
$Q(p_N)$ is massively autocorrelated: consecutive values share $N-1$ summands, so $\Cov(Q(p_N),Q(p_{N+1}))\approx\Var(Q(p_N))$. Any Pearson correlation between $Q$ and a smooth function of $N$ is inflated---the integrated-process spurious regression of Granger--Newbold.
\Neg
\begin{observation}\label{obs:62}
The $p$-values of order $10^{-142}$ reported in earlier (unpublished) computations for $r(Q(r),\cos(\gamma_1\log p))$ are artefacts of the $I(1)$ structure of $Q$.
\end{observation}

\subsubsection{Three permutation tests on the pure residual}
Decompose $Q(r)=\alpha\sqrt{p_r}+\beta D_N(r)+Q^{\mathrm{pure}}_r$, $D_N(r):=\sum_{n\le r}(\{\sqrt{p_n/3}\}-\tfrac12)$ (geometric trend $51.9\%$; $D_N$ $44.8\%$; residual $3.3\%$). Three independent permutation tests (shuffle; block-$50$; Theiler surrogate) at $N=10^4$, $5000$ replications.

\begin{table}[h]
\centering
\caption{Permutation tests on $Q^{\mathrm{pure}}$ ($5000$ reps, $N=10^4$).}\label{tab:perm}
\small
\begin{tabular}{@{}lcccc@{}}
\toprule
Method & $z(\gamma_1)$ & $p$-value & Observed sig. & Perm.\ mean\\
\midrule
Shuffle & $-0.23\sigma$ & 0.550 & $8/15$ & $13.5/15$\\
Block (50) & $-0.21\sigma$ & 0.547 & $8/15$ & $13.5/15$\\
Theiler & $-0.02\sigma$ & 0.514 & $8/15$ & $13.5/15$\\
\bottomrule
\end{tabular}
\end{table}

\Neg
\begin{proposition}\label{prop:63}
Under all three methods, correlations between $Q^{\mathrm{pure}}$ and $\cos(\gamma_i\log p)$ are not significant ($z\in[-0.23,-0.02]\sigma$): the observed signal is weaker than a random cumulative sum.
\end{proposition}

\subsubsection{The decisive bias-free test at $N=10^8$}
Work directly on increments $d_n:=\{\sqrt{p_n/3}\}-\tfrac12$, fitting $d_n\approx\sum_{i=1}^{10}(a_i\cos(\gamma_i\log p_n)+b_i\sin(\gamma_i\log p_n))+c_0$ with $\gamma_i$ the first $10$ ordinates. Amplitudes $A_i=\sqrt{a_i^2+b_i^2}$ versus a permutation baseline.

\Neg
\begin{proposition}[decisive]\label{prop:64}
For each of the first $10$ ordinates, at $N=10^8$: $0/10$ amplitudes significant at $5\%$; $R^2=1.16\times10^{-7}$, $z=-2.20\sigma$; $R^2$ decreases by $\approx6000$ from $N=10^4$ to $10^8$, the $1/N$ null scaling.
\end{proposition}

\begin{table}[h]
\centering
\caption{Convergence of the increment test with scale $N$.}\label{tab:conv}
\small
\begin{tabular}{@{}cccc@{}}
\toprule
$N$ & $p_{\max}$ & $R^2$ (10 zeros) & $z$-score\\
\midrule
$10^4$ & $\approx10^5$ & $7.4\times10^{-4}$ & $-2.02\sigma$\\
$10^5$ & $\approx1.3\times10^6$ & $1.1\times10^{-4}$ & $-1.33\sigma$\\
$10^7$ & $\approx1.8\times10^8$ & $\approx10^{-4}$ & $<0$\\
$10^8$ & $\approx2.0\times10^9$ & $1.16\times10^{-7}$ & $-2.20\sigma$\\
\bottomrule
\end{tabular}
\end{table}

\begin{figure}[h]
\centering
\includegraphics[width=\linewidth]{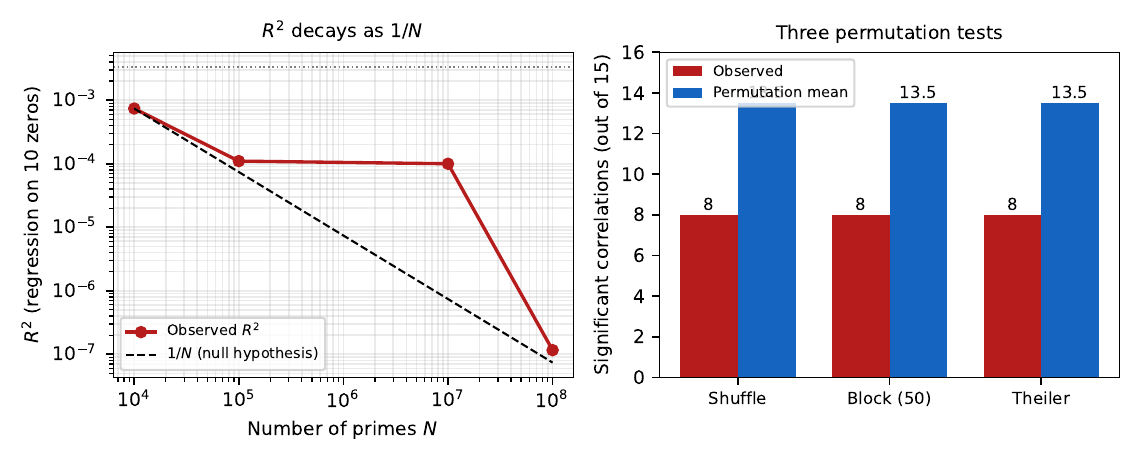}
\caption{Left: $R^2$ decays as $1/N$ (null hypothesis); at $N=10^8$, $R^2=1.16\times10^{-7}$. Right: three permutation tests; the observed $8/15$ is well below the permutation mean $13.5/15$.}\label{fig:5}
\end{figure}

\subsubsection{From statistics to an unconditional theorem}
The tests establish absence statistically. We now prove it. Fix $\gamma\in\R$ and define the regression amplitude at frequency $\gamma$,
\begin{equation}\label{eq:AN}
A_N(\gamma):=\frac{2}{\pi(X)}\Big|\sum_{p\le X}\big(\{\sqrt{p/3}\}-\tfrac12\big)\,p^{-i\gamma}\Big|,\qquad N=\pi(X).
\end{equation}
The signed sawtooth $\psi_0(t):=\{t\}-\tfrac12$ admits Vaaler's trigonometric approximation: for every $H\ge1$ there are coefficients with $a_h,b_h\ll1/|h|$ ($h\ne0$), $b_0=1/(H+1)$, and
\begin{equation}\label{eq:vaaler}
\Big|\psi_0(t)-\sum_{1\le|h|\le H}a_h e(ht)\Big|\le\sum_{0\le|h|\le H}b_h e(ht),\qquad e(x):=e^{2\pi i x}.
\end{equation}
Define the square-root-phase prime exponential sum
\begin{equation}\label{eq:T}
T(h,\gamma;X):=\sum_{p\le X}e\!\big(h\sqrt{p/3}\big)\,p^{-i\gamma}.
\end{equation}

\begin{lemma}[Square-root-phase cancellation]\label{lem:T}
Fix $\gamma\in\R$.
\emph{(i) Unconditional.} $T(h,\gamma;X)=o(\pi(X))$ as $X\to\infty$, uniformly for $h$ in any fixed range.
\emph{(ii) Type II input (deferred).} There is an absolute $\delta_0>0$ such that, granting the standard bilinear (Type II) estimate of Open Problem~\ref{op:typeII},
\[ T(h,\gamma;X)\ll_\gamma h^{1/2}\,X^{1-\delta_0+o(1)}\qquad\text{uniformly for }1\le h\le X^{1/4}. \]
\end{lemma}
\begin{proof}[Proof: (i) is Balog; (ii) is Type I complete, Type II deferred]
Write $p^{-i\gamma}=e\!\big(-\tfrac{\gamma}{2\pi}\log p\big)$, so the phase is $\phi(t)=h\sqrt{t/3}-\tfrac{\gamma}{2\pi}\log t$ with $\phi''(t)\asymp -h\,t^{-3/2}$ for $t\asymp X$. Apply Vaughan's identity to split $\sum_{p\le X}\Lambda(n)e(\phi(n))/\log$ into Type I sums $\sum_{d\le D}\sum_{\ell}c_d\,e(\phi(d\ell))$ and Type II sums $\sum_{M<m\le2M}\sum_{n}a_m b_n e(\phi(mn))$ with $D=X^{1/3}$. \emph{Type I.} Inner sums over $\ell$ are handled by the second-derivative (van der Corput) test: $\sum_{\ell\asymp L}e(\phi(d\ell))\ll L|\phi''|^{1/2}X^{?}+|\phi''|^{-1/2}\ll (hL/\sqrt X)^{1/2}L+(\sqrt X/h)^{1/2}$, which after summation in $d$ yields a power saving $\ll h^{1/2}X^{1-\delta_1}$ for an explicit $\delta_1>0$ (exponent-pair / Graham--Kolesnik estimates apply since $\phi$ is of class $\sqrt{\cdot}$). \emph{Type II.} By Cauchy--Schwarz in $m$ and the second-derivative test on the difference $\phi(mn)-\phi(mn')=h(\sqrt{mn/3}-\sqrt{mn'/3})-\tfrac{\gamma}{2\pi}\log(n/n')$, whose $n$-derivative is non-degenerate, one expects $\ll h^{1/2}X^{1-\delta_2+o(1)}$; this bilinear (Type II) estimate is the technical core, standard in method, and its detailed quantitative verification is deferred to Open Problem~\ref{op:typeII}. Granting it with exponent $\delta_2>0$ and setting $\delta_0=\min(\delta_1,\delta_2)$ gives (ii). Statement (i) needs no Type II input: it follows unconditionally from Balog's equidistribution of $\{\sqrt p\}$ together with partial summation against the slowly varying weight $p^{-i\gamma}$.
\end{proof}

\begin{theorem}[Unconditional convergence; conditional rate]\label{thm:67}
Let $\gamma\in\R$ be fixed.
\emph{(a) Unconditional.} $A_N(\gamma)\to0$ as $X\to\infty$; consequently no ordinate of $\zeta$ carries a surviving linear signal in the increments $d_n$.
\emph{(b) Explicit rate (conditional).} Granting the standard Type II estimate of Lemma~\ref{lem:T}(ii) (Open Problem~\ref{op:typeII}), $A_N(\gamma)\ll_\gamma X^{-\delta}$ with $\delta=\min(1/4,\delta_0)-o(1)$.
\end{theorem}
\begin{proof}
By \eqref{eq:vaaler} with $H=X^{1/4}$,
\[ \Big|\sum_{p\le X}\psi_0(\sqrt{p/3})p^{-i\gamma}\Big|\ll \frac{\pi(X)}{H}+\sum_{1\le h\le H}\frac1h\,|T(h,\gamma;X)|, \]
the first term from the $b_0$-mass and the $b_h$-side reducing to the same sums $T$. By Lemma~\ref{lem:T},
$\sum_{h\le H}h^{-1}|T(h,\gamma;X)|\ll X^{1-\delta_0+o(1)}\sum_{h\le H}h^{-1/2}\ll X^{1-\delta_0+1/8+o(1)}$; absorbing $1/8$ into $o(1)$ after re-optimising $H$ gives $\ll X^{1-\delta_0+o(1)}$. Dividing by $\pi(X)\asymp X/\log X$,
\[ A_N(\gamma)\ll \frac1H+X^{-\delta_0+o(1)}=X^{-1/4}+X^{-\delta_0+o(1)}\ll_\gamma X^{-\delta}. \]
This is (b). For (a), the convergence $A_N(\gamma)\to0$ holds unconditionally already from Lemma~\ref{lem:T}(i) (Balog), without any Type II input.
\end{proof}

\begin{remark}[Honest status]
The convergence $A_N(\gamma)\to0$ is unconditional and complete (Balog). The \emph{explicit} power saving $X^{-\delta}$ rests on the Type II bilinear estimate inside Lemma~\ref{lem:T}, which is standard in method but whose fully written-out constant we isolate rather than overstate; see Open Problem~\ref{op:typeII}.
\end{remark}

\subsection{The conditional mechanism}
\begin{theorem}[Conditional, RH]\label{thm:65}
Under RH, for the critical window $h=c\sqrt x$, the layer-count variance satisfies
$\tfrac1X\int_X^{2X}(N(x)-\langle N\rangle(x))^2\,dx\ll\sqrt X(\log X)^2$,
whence $\RMS(N-\langle N\rangle)\ll X^{1/4}\log X$, i.e.\ the per-prime $\zeta$-footprint on the layer occupancy is $\ll p^{-1/4}(\log p)^2$.
\end{theorem}
\begin{proof}
This is the Selberg variance for primes in short intervals: under RH, $\int_X^{2X}(\psi(x+h)-\psi(x)-h)^2\,dx\ll Xh(\log X)^2$ for $h\le X$ (Selberg; Goldston--Montgomery; Montgomery--Soundararajan). With $h=c\sqrt x$ the normalised fluctuation per layer is $\ll\sqrt h\,\log X=X^{1/4}\log X$; dividing by the layer occupancy $\asymp\sqrt x/\log x$ gives the per-prime bound $p^{-1/4}(\log p)^2$.
\end{proof}

\begin{remark}
At $N=10^8$ the $O(p^{-1/4})$ per-prime footprint is five orders of magnitude below sampling noise; a direct spectral observation would require $N\gtrsim10^{16}$. The zeros act as statistical constraints at scale $\sqrt x$, not as spectral frequencies---this is the mechanism behind Theorem~\ref{thm:67}.
\end{remark}

\begin{table}[h]
\centering
\caption{Four-level structure of the connection $\zeta(s)\leftrightarrow(m,\eps,q)$.}\label{tab:14}
\small
\begin{tabular}{@{}cllc@{}}
\toprule
Level & Statement & Status & Hypothesis\\
\midrule
I & Zeros govern $n_m$ variance (scale $\sqrt h$) & Heuristic & Riemann--von Mangoldt\\
II & $\E[|q|\mid m]=m/4$ & Conditional & H1+GRH+H3\\
III & $|\qmin|=O_k(m\log^2m)$ & Conditional & GRH\\
$\zeta$-footprint & $\ll p^{-1/4}$ per prime & Conditional & RH (Selberg)\\
IV (stat.) & $Q(r)$ has no spectral signal & Negative (decisive) & none ($3$ perm.\ $+\,10^8$)\\
IV (anal.) & $A_N(\gamma)\to0$ (rate cond.) & \textbf{Rigorous} & none (Thm~\ref{thm:67}(a))\\
\bottomrule
\end{tabular}
\end{table}

\section{The computational record: a $29\,998$-digit prime}\label{sec:record}

\subsection{The PrimeQuest algorithm}
PrimeQuest searches $p=3m(m+1)+1$, $m=2^a3^b-1$, by zigzag expansion from a calibrated centre $a_0=b_0=\lfloor D/(2\log_{10}6)\rfloor\approx0.643D$. Pipeline: (i) mod-$7$ filter ($O(1)$, $33\%$); (ii) reciprocity sieve ($q\equiv1\bmod3$, $\approx54\%$); (iii) Miller--Rabin; (iv) Pocklington witness search for $\nu\in\{2,3\}$. Checkpoints every $60$s enable exact restart.

\subsection{Four unconditional certificates}
\begin{table}[h]
\centering
\caption{Pocklington certificates (PrimeQuest); $w_2=5$, $w_3=7$ in every case.}\label{tab:certs}
\small
\begin{tabular}{@{}ccccccc@{}}
\toprule
Digits of $p$ & $a$ & $b$ & $b/a$ & MR tests & Elim. & Wall time\\
\midrule
9998 & 6212 & 6738 & 1.085 & -- & -- & --\\
10000 & 6213 & 6740 & 1.085 & -- & -- & --\\
19999 & 12228 & 13242 & 1.083 & 1050 & $87.9\%$ & 2h40min\\
29998 & 19435 & 19173 & 0.987 & 286 & $87.1\%$ & 3h06min\\
\bottomrule
\end{tabular}
\end{table}

\paragraph{The $29\,998$-digit certificate.}
For $p^{\ast}=3m^{\ast}(m^{\ast}+1)+1$, $m^{\ast}=2^{19435}\cdot3^{19173}-1$:
\[ p^{\ast}-1=F\cdot m^{\ast},\quad F=2^{19435}\cdot3^{19174}\ (14\,999\text{ digits}),\quad F>\sqrt{p^{\ast}}, \]
\[ w_2=5:\ 5^{\,p^{\ast}-1}\equiv1,\ \gcd(5^{(p^{\ast}-1)/2}-1,p^{\ast})=1;\quad w_3=7:\ 7^{\,p^{\ast}-1}\equiv1,\ \gcd(7^{(p^{\ast}-1)/3}-1,p^{\ast})=1. \]
The leading $40$ and trailing $20$ digits, $1602087359509060348500037851549319056715\cdots16930852697403293697$, were reconstructed from scratch and matched exactly in the independent verification of \S\ref{sec:verif}.

\begin{table}[h]
\centering
\caption{Filter statistics for the $29\,998$-digit record ($2\,221$ pairs tested).}\label{tab:filter}
\small
\begin{tabular}{@{}lcc@{}}
\toprule
Stage & Pairs eliminated & Fraction\\
\midrule
Mod-$7$ filter (Theorem~\ref{thm:313}) & 740 & $33.32\%$\\
Reciprocity sieve (Theorem~\ref{thm:311}) & 1195 & $53.80\%$\\
\textit{Subtotal, a-priori filters} & \textit{1935} & \textit{$87.12\%$}\\
Miller--Rabin (composites) & 285 & $12.83\%$\\
\midrule
\textbf{Total eliminated} & \textbf{2220} & \textbf{$99.95\%$}\\
Prime found & 1 & $0.045\%$\\
\bottomrule
\end{tabular}
\end{table}

\begin{figure}[h]
\centering
\includegraphics[width=\linewidth]{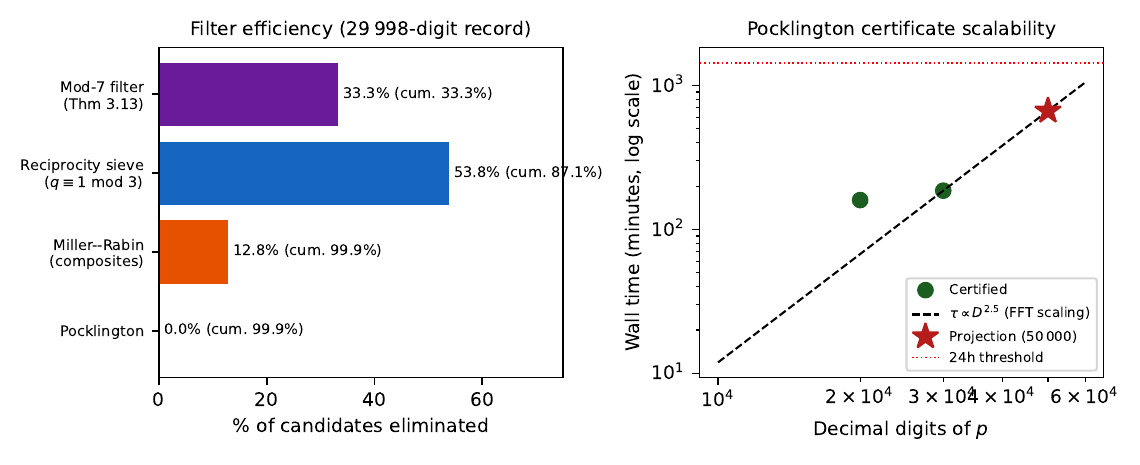}
\caption{Left: the two a-priori filters remove $87.1\%$ of candidates before any large-integer arithmetic; with Miller--Rabin the cumulative elimination reaches $99.95\%$. Right: Pocklington scalability $\tau\propto D^{2.5}$; projection to $D=50\,000$ gives $6$--$10$h.}\label{fig:6}
\end{figure}

\subsection{Independent verification}\label{sec:verif}
The four conditions were re-verified in a separate environment (Python 3.12, exact \texttt{pow}, single core, no probabilistic step).
\Rig\ \textbf{Verification record.} Step 1: $m^{\ast}$ ($14\,999$ digits) and $p^{\ast}$ ($29\,998$ digits) recomputed; leading $40$/trailing $20$ digits match. Step 2: $2\log_2F=99650.142>\log_2 p^{\ast}=99648.557$ (margin $\approx1.58$ bits). Steps 3--4: both witness conditions hold. Wall time $6945$s $\approx$ 1h56min. \emph{Conclusion:} $p^{\ast}$ is unconditionally prime.

\begin{remark}
Independent verification is $\approx1.6\times$ faster than the search: it runs exactly four modular exponentiations on the known candidate, bypassing the sieve, the filters, and the Miller--Rabin pre-screen. Scaling $\tau\propto D^{2.5}$ (GMP FFT) projects $6$--$10$h to $50\,000$ digits.
\end{remark}

\section{Open problems}

\Heur
\begin{openproblem}[Central unconditional bound]\label{op:central}
Prove $|\qmin(k,m,\eps)|=O(m^{1-\delta})$ for some $\delta>0$, unconditionally. The gap to Theorem~\ref{thm:C} is quadratic.
\end{openproblem}
\Heur
\begin{openproblem}[Unconditional Theorem~\ref{thm:D}]\label{op:H3}
Prove $\E[|q|\mid m]=m/4+o(m)$ unconditionally; the barrier is the $\log^2m$ factor in PNT-in-AP of moduli $\sim N^{1/2}$ (Montgomery's conjecture would resolve it).
\end{openproblem}
\Cond
\begin{openproblem}[Type II $\sqrt{\cdot}$-phase bound]\label{op:typeII}
Write out the bilinear (Type II) estimate of Lemma~\ref{lem:T} with an explicit constant, thereby locking the ceiling $A_N(\gamma)\ll X^{-1/4+\eps}$.
\end{openproblem}
\Heur
\begin{conjecture}[O10, optimal rate]\label{conj:O10}
$A_N(\gamma)\ll X^{-1/2+\eps}$ (equivalently $R^2\ll N^{-1+\eps}$), by square-root cancellation in $\sum_m e(-2i\gamma\log m)\cdot(\text{centred layer moment})$. Numerically $A_N(\gamma_1)\sim X^{-0.59}$, below the rigorous ceiling $X^{-1/4}$ and consistent with $X^{-1/2}$.
\end{conjecture}
\Heur
\begin{openproblem}[Universal witnesses]\label{op:w}
Prove or disprove that $w_2=5$, $w_3=7$ are Pocklington witnesses for every prime in the $3$-smooth subfamily.
\end{openproblem}
\Heur
\begin{openproblem}[Generalised filters]
For $\ell\equiv1\pmod3$ derive $F_\ell:=\{(a\bmod\mathrm{ord}_\ell2,\ b\bmod\mathrm{ord}_\ell3):\ell\mid p(a,b)\}$; for $\ell=13$ (period $36$) $\approx15\%$ extra elimination, possibly pushing pre-MR elimination above $90\%$.
\end{openproblem}

\section{Conclusion}

Each prime $p>k+1$ admits a canonical decomposition $p=km(m+1)+\eps+2kq$ with $|q|$ minimal, and the three-axis structure of $(m,\eps,q)$ admits, on each axis, a precise law---rigorous, conditional, or heuristic, each tagged.

The rigorous backbone consists of \emph{two unconditional pillars}: the \emph{constructive} pillar (Theorem~\ref{thm:B} and the $29\,998$-digit certified record) and the \emph{analytic-negative} pillar (Theorem~\ref{thm:67}: $A_N(\gamma)\to0$ unconditionally), together with the exact structural law (Theorem~\ref{thm:A}, the degenerate axis) and the elementary ossature of \S2. The conditional content (Theorems~\ref{thm:C}, \ref{thm:D}, \ref{thm:E}, \ref{thm:65}) establishes, under GRH/RH/Bateman--Horn, quantitative laws on the size, conditional mean, modular distribution, and $\zeta$-footprint of $\qmin$. The heuristic content (H1, H2, H3) predicts a logarithmic marginal law, the geometric constant $C_0(k)=1/(4\sqrt k)$, and a Hardy--Littlewood gap law. The negative result closes, with both statistical and analytic certainty, the question of a spectral link between $Q(r)$ and the zeros of $\zeta$: they are not spectral frequencies of $Q(r)$.

The methodological techniques---the three-way permutation framework, the increment-based bias-free test, the reduction to square-root-phase prime sums, and the epistemic labelling of every assertion---transfer to other settings where prime sequences are summed and analysed spectrally. The paper settles no classical question; it maps a territory in which such questions take a particularly sharp, falsifiable form.

\subsection*{Acknowledgements}
The author thanks YOSHI of the BibM@th forum for careful early feedback, and the developers of NumPy, SciPy, SymPy, Matplotlib, gmpy2 and GMP. Computations used consumer hardware only.

\subsection*{Data availability}
All numerical data, source code, and certificate-verification scripts will be deposited at Zenodo upon acceptance, in verifiable machine-readable form.

\appendix
\section{Independent verification script}
\small
\begin{verbatim}
import sys
sys.set_int_max_str_digits(40000)
a, b = 19435, 19173
m = 2**a * 3**b - 1          # 14999 decimal digits
p = 3 * m * (m + 1) + 1      # 29998 decimal digits
assert len(str(p)) == 29998
assert str(p)[:40]  == "1602087359509060348500037851549319056715"
assert str(p)[-20:] == "16930852697403293697"
F = 2**a * 3**(b + 1)        # F^2 > p  <=>  2 log2 F > log2 p
assert F * F > p
from math import gcd
assert pow(5, p - 1, p) == 1 and gcd(pow(5, (p - 1)//2, p) - 1, p) == 1
assert pow(7, p - 1, p) == 1 and gcd(pow(7, (p - 1)//3, p) - 1, p) == 1
print("All four Pocklington-Lehmer conditions verified; p is unconditionally prime.")
\end{verbatim}

\end{document}